\documentclass[11pt,a4paper]{amsart}
\usepackage[english]{babel}
\usepackage{amsmath,amsthm, amssymb, amsfonts}
\usepackage{setspace}
\usepackage{anysize}
\usepackage{url}
\usepackage{bm}
\usepackage{color,graphicx}
\usepackage{paralist}
\usepackage{verbatim}
\usepackage[utf8]{inputenc}
\usepackage{enumitem}
\usepackage{soul}

\usepackage{hyperref}
\definecolor{darkblue}{RGB}{0,0,160}
\hypersetup{
colorlinks,%
citecolor=black,%
filecolor=black,%
linkcolor=darkblue,%
urlcolor=darkblue
}

\theoremstyle{plain}
\newtheorem{theorem}{\bf Theorem}[section]

\newtheorem*{theorem*}{Theorem}
\newtheorem{proposition}[theorem]{\bf Proposition}
\newtheorem{lemma}[theorem]{\bf Lemma}

\newtheorem*{conjecture*}{\bf Conjecture}
\theoremstyle{definition}
\newtheorem{definition}[theorem]{\bf Definition}
\newtheorem{remark}[theorem]{\bf Remark}
\newtheorem{notation}[theorem]{\bf Notation}
\newtheorem{method}[theorem]{\bf Method}
\theoremstyle{remark}
\newtheorem{example}[theorem]{\bf Example}

\renewcommand\>{\rangle}
\newcommand\<{\langle}

\def \type{{\operatorname{type}}}

\def \depth{{\operatorname{depth}}}

\def \lk{{\operatorname{link}}}

\def \reg{{\operatorname{reg}}}
\def \pd{{\operatorname{pd}}}
\def \rk{{\operatorname{rk}}}

\def \G{{\Gamma}}

\def \aaa{{\bf a}}

\def \ZZ{\mathbb Z}
\def \D{\Delta}
\def \Dc{\Delta^{\mathrm{c}}}

\def \Ds{^{\operatorname{si\!}}\Delta}

\def \gin{\operatorname{gin}}
\def\kk{\Bbbk}

\begin{document}

\title{Generic and special constructions of pure $O$-sequences}

\author[A.~Constantinescu]{Alexandru Constantinescu}
\address{Mathematisches Institut, Freie Universität Berlin, Arnimallee
  3, 14195 Berlin, Germany}
\email{aconstant@math.fu-berlin.de}
\urladdr{\href{http://userpage.fu-berlin.de/aconstant/Main.html}{http://userpage.fu-berlin.de/aconstant/Main.html}}

\author[T.~Kahle]{Thomas Kahle}
\address {Fakultät für Mathematik, Otto-von-Guericke Universität,
  Universitätsplatz 2, 39106 Magdeburg, Germany.}
\email{thomas.kahle@ovgu.de}
\urladdr{\href{http://www.thomas-kahle.de}{http://www.thomas-kahle.de}} 

\author[M.~Varbaro]{Matteo Varbaro}
\address{ Dipartimento di Matematica,
 Universit\`a di Genova,
 Via Dodecaneso 35, Genova 16146, Italy}
\email{varbaro@dima.unige.it}
\urladdr{\href{http://www.dima.unige.it/~varbaro/}{http://www.dima.unige.it/~varbaro/}} 

\thanks{The research for this project was carried out during stays of
  the authors at CIRM, Trento and MSRI, Berkeley}

\date{{\small \today}}



\begin{abstract}
  It is shown that the $h$-vectors of Stanley-Reisner rings of three
  classes of matroids are pure $O$-sequences.  The classes are
  (a)~matroids that are truncations of matroids, or more generally of
  Cohen-Macaulay complexes, (b)~matroids whose dual is
  $(\text{rank}+2)$-partite, and (c)~matroids of Cohen-Macaulay type
  at most five.  Consequences for the computational search for a
  counterexample to a conjecture of Stanley are discussed.
\end{abstract}

\maketitle
\section*{Introduction}
The $f$-vector and $h$-vector are fundamental invariants of a
simplicial complex, encoding the number of faces that the complex has
in each dimension.  What can be said in general about these vectors?
Starting from Euler's polyhedron formula in the middle of the 18th
century, different conditions and eventually characterizations have
been found.  It seems natural to ask for a description of the set of
$f$- or equivalently $h$-vectors of all simplicial complexes or all
pure simplicial complexes in a given dimension.  The situations for
these two classes are quite different.  There is a precise
characterization of the set of $f$-vectors of all simplicial complexes
due to Schützenberger, Kruskal, and
Katona~\cite[Theorem~II.2.1]{St96}.  The opposite is the case for pure
simplicial complexes---a characterization is believed to be
intractable.  As Ziegler points out, it would \emph{solve all basic
  problems in design theory}
\cite[Exercise~8.16]{ziegler00LecturesPolytopes}.  The celebrated
$g$-theorem characterizes $h$-vectors of simplicial polytopes
(\cite{billera1980sufficiency,billera1981proof,stanley1980number}) and
it is conjectured that this characterization also applies to
simplicial spheres (of which there are many more than boundaries of
simplicial polytopes~\cite{kalai88:_many}).  This indicates that
subclasses of pure complexes---like Gorenstein, Cohen-Macaulay, or
matroid complexes---may be feasible.  It is known for a long time,
essentially due to Macaulay, that the sets of vectors that arise as
$h$-vectors of Cohen-Macaulay complexes consist exactly of
\emph{$O$-sequences}---Hilbert functions of Artinian
algebras~\cite{macaulay1927some}.  Although necessary conditions are
known, characterizations for matroid or Gorenstein complexes are open
and may be out of reach.

In this paper we focus on matroids.  They were originally introduced
by Whitney as a way to study the concept of independence~\cite{whi35}.
Subsequently they appeared in a wide range of mathematical areas from
linear algebra, (real) algebraic geometry, and combinatorial geometry
to graph theory, optimization, and approximation theory.  The new
edition of Oxley's book~\cite{Ox11} provides an excellent guide to the
theory.  Interest in algebraic properties of matroids is still growing
as witnessed by recent work of DeConcini-Procesi~\cite{DP11},
Holtz-Ron~\cite{HR11}, Lenz~\cite{Le11b}, Moci~\cite{Mo12}, and
Huh~\cite{Hu12b,Hu12a}.

What properties should the $h$-vector of a matroid have?  Since
matroids are Cohen-Macaulay, their $h$-vectors must be $O$-sequences.
In \cite{S77} Stanley shows that they are also Hilbert functions of
Artinian algebras whose socle is concentrated in one degree.  He
conjectured that for any matroid one can even find a \emph{monomial}
algebra with this property.  In this case, its Hilbert function is
called a \emph{pure $O$-sequence}.
\begin{conjecture*}[{\cite[p.59]{S77}}]
  The $h$-vector of a matroid complex is a pure $O$-sequence.
\end{conjecture*}

For an abstract simplicial complex $\Delta$ on $[n] := \{1,\dots,
n\}$, let $f_{i}(\Delta)$ be the number of faces of size~$i$. Let $d =
\max\{i : f_i\neq 0\}$ be the \emph{rank} of~$\D$.  The vector
$f=(f_{0},\dots,f_{d})$ is the \emph{f-vector} of~$\Delta$.  It
encodes the same information as the \emph{h-vector} $h(\Delta)
=(h_{0},\dots,h_{s})$ whose component $h_{i}$ is the coefficient of
$x^{d-i}$ in the polynomial $\sum_{i=0}^{d}f_{i}(x-1)^{d-i}$.  A
central tool for the study of the $h$-vector is the
\emph{Stanley-Reisner ring $\kk[\D] := \kk[x_{1},\dots,x_{n}] /
  I_{\D}$}, where $I_{\D} = \left(\prod_{i\in G}x_{i} : G \notin \D
\right)$ is the \emph{Stanley-Reisner ideal}.  In this setting, the
$h$-vector appears as the coefficient vector of the numerator
polynomial of the Hilbert series of~$\kk[\Delta]$ (see~\cite{St96}).
The field $\kk$ in this definition is arbitrary, and homological
properties of $\kk[\Delta]$ may depend on the characteristic.
However, Stanley's conjecture is field independent.

The problem raised by Stanley is extremely difficult and the authors are not strong
believers in the validity of the conjecture.  The complications are in
part due to the strange properties of pure $O$-sequences.  For
instance, they need not be unimodal, and it is likely that they can
not be characterized well~\cite{BMMNZ}.  On the positive side, it is
known that both pure $O$-sequences and $h$-vectors of matroid
complexes satisfy a common set of inequalities
\cite{CH97,Hi98}:
\begin{equation*} h_{0}\leq h_{1}\leq\dots\leq h_{\lfloor\frac{s}{2}\rfloor},
 \qquad h_i\le h_{s-i} \text{ for } 0\le i\le\lfloor\frac{s}{2}\rfloor.
\end{equation*}
In contrast, the Brown-Colbourn inequalities 
\begin{equation*}
  \text{for any $b \geq 1$ } \quad (-1)^{j} \sum_{i=0}^{j}(-b)^{i}h_{i}
  \geq 0, \quad 0 \leq j \leq s.
\end{equation*}
hold for $h$-vectors of matroids, but not pure
$O$-sequences~\cite{brown1992roots}.  Other than this our
understanding is poor.  Positive answers to Stanley's conjecture are
known for short $h$-vectors~\cite{DKK,HSZ}, and for special classes of
matroids~\cite{Me01,MNRV12,Oh12}.  In the present paper we prove that
Stanley's conjecture holds for matroids that are truncations of other
matroids and for matroids whose $h$-vector $(1,h_{1},\dots, h_{s})$
satisfies $h_s\le 5$ (with no restriction on $s$).  We employ two
completely different methods of proof, both of which have potential
for generalizations. As a consequence of our results, the search for
counterexamples is pushed closer to today's computational limits.

\subsection*{\texorpdfstring{Generic pure $O$-sequences}{Generic pure O-sequences}}
The Stanley-Reisner ring $\kk[\Delta]$ of a matroid $\D$ is level.  To
produce a pure $O$-sequence which equals the $h$-vector of $\Delta$ it
would suffice to pass to a \emph{monomial} Artinian reduction.
Unfortunately, a monomial ideal rarely has one.  In this context, the
generic initial ideal may come to mind.  It has the same $h$-vector as
the original ideal and (in characteristic zero) is strongly stable.
Therefore it possesses a regular sequence of variables and a monomial
Artinian reduction.  However, this does not prove Stanley's conjecture
as typically the quotient modulo the generic initial ideal is not
level.  We envision an approach to Stanley's conjecture in which one
interpolates between these two objectives with a less drastic version
of the generic initial ideal (Remark~\ref{rem:lgin}).  In
Section~\ref{sec:gin} we study this genericity of matroids and show
that a generalization of Stanley's conjecture holds for all simplicial
complexes that are truncations (skeletons) of matroids
(Theorem~\ref{thm:StanleyForTruncation}).

\subsection*{\texorpdfstring{Special pure $O$-sequences}{Special pure O-sequences}}
In matroid theory duality is central.  If $\Delta$ is a matroid, then
the complex $\Dc$ whose facets are the complements of facets of
$\Delta$ is the \emph{dual matroid}.  Directly from the definitions,
its Stanley-Reisner ideal $I_{\Dc}$ equals the \emph{cover
  ideal} $J(\Delta)$ of~$\Delta$.  In this paper $h_{\Delta}$ is the
$h$-vector of (the quotient by) $I_{\Delta}$ and $h^{\Delta}$ that of
(the quotient by) $J(\Delta)$.  By matroid duality it suffices to
prove Stanley's conjecture for either of the classes.  Several known
results on matroid complexes are stated in terms of the dual
matroid~\cite{DKK,Me01,Oh12}, which may be taken as an indication that
that the cover ideal is a natural object.  This perspective permeates
the work of the first and third author and also our
Section~\ref{sec:partite}, where we aim at a generalization of the
construction of pure $O$-sequences in~\cite{CV3}.  This construction
is recursive and relies on finding pure $O$-sequences for links and
deletions in the matroid.  When trying to generalize the construction
we require a \emph{compatibility condition}
(Lemma~\ref{lem:compatibility} and Definition~\ref{def:partition
  comp}) the checking of which remains an obstacle.  Carefully keeping
track of the contributions in the recursion allows us to prove
Stanley's conjecture for duals of matroids with at most $\text{rank} +
2$ parallel classes (Theorem~\ref{thm:d2}).  Exploiting the
constraints on the $h$-vectors of matroids whose dual has a fixed
number of parallel classes, proved in \cite{CV3}, we can show
Stanley's conjecture when the type is at most five
(Theorem~\ref{thm:type5}).

\subsection*{The search for a counterexample}

Matroids on nine or fewer elements have been enumerated by Mayhew and
Royle~\cite{MR} and Stanley's conjecture has been confirmed for all of
them in~\cite{DKK}.  Beyond nine vertices, mostly due to the lack of a
good list of candidates, only sporadic experiments have been carried
out.  Our results have implications for the search for a
counterexample.  By Theorem~\ref{thm:type5}, any candidate
counterexample must be of Cohen-Macaulay type at least six.  To
confirm such a counterexample in silico would include enumeration of
all $\binom{N}{6}$ socles where $N$ is a binomial coefficient (see
Example~\ref{ex:doubleBinomial}).  The methods of
Section~\ref{sec:partite}, in particular
Lemma~\ref{lem:compatibility}, imply faster searches for pure
$O$-sequences realizing the $h$-vector of the cover ideal of a given
matroid.  In Section~\ref{sec:computer} we discuss our computational
efforts.  As part of this project we developed a small
\verb!C++!-library which can be used to enumerate pure $O$-sequences
The source code is available at
\url{https://github.com/tom111/GraphBinomials} and is licensed under
the~GPL.  We also made intensive use of Cocoa~\cite{CocoaSystem},
Macaulay2~\cite{M2} and Sage~\cite{sage}.

\subsection*{Acknowledgement}
We are grateful to Fondazione Bruno Kessler for funding a two week
stay at CIRM Trento where this project was started, and to MSRI, where
this paper took shape.  We thank Alex Fink for valuable suggestions
and comments.  

\section{Linear resolutions and the generic initial ideal}
\label{sec:gin}
Let $S=\Bbbk[x_1,\ldots ,x_n]$ be a polynomial ring over a field
$\kk$.  For any ideal $I\subseteq S$ we denote $\gin(I)$ the
\emph{generic initial ideal} with respect to the graded reverse
lexicographic term order.  Any graded $S$-module $M$ has a minimal
graded free resolution:
\[
0\rightarrow F_p\xrightarrow{\delta_{p}}F_{p-1}\rightarrow \ldots
\rightarrow F_1 \xrightarrow{\delta_1}
F_0\xrightarrow{\delta_0}M\rightarrow 0,
\] 
in which $F_i=\bigoplus_{j\in\ZZ}S(-j)^{\beta_{i,j}(M)}$.  Let
$Z_i(M)=\ker \delta_{i}$ be the $i$th syzygy module of~$M$. The module $M$
has a \emph{$k$-linear resolution} if $\beta_{i,j}(M)= 0$ whenever
$j\neq i+k$. It is \emph{componentwise linear} if $M_{\langle
  k\rangle}$ has $k$-linear resolution for all $k\in\ZZ$, where
$M_{\langle k\rangle}$ is the submodule of $M$ generated by all
homogeneous elements of degree~$k$.  It is not difficult to show that,
if $M$ has a linear resolution, then it is componentwise linear, for
example, using \cite[Corollary~2.5]{CH03}.
Linearity of the free resolution is a genericity condition.  This
intuition is justified by
\begin{theorem}[{\cite[Theorem~1.1]{AHH}}]\label{thm:AHH}  Let
  $\text{char}(\kk) = 0$.  An ideal $I\subset S$ is componentwise
  linear if and only if $\beta_{i,j}(S/I)=\beta_{i,j}(S/\gin(I))$ for
  all~$i,j$.
\end{theorem}

Since $I = Z_{0} (S/I)$, one may ask which conclusions are implied if
$Z_{i}(S/I)$ is componentwise linear.  The following result gives one
direction.

\begin{proposition}[{\cite[Theorem~5.7]{caviglia2013componentwise}}]\label{prop:CS}
  Let $I\subset S$ be a graded ideal such that
  $\beta_{i,j}(S/I)=\beta_{i,j}(S/\gin(I))$ for all $i>s+1$ and
  $j\in\ZZ$. Then $Z_s(S/I)$ is componentwise linear.
\end{proposition}

In general, the other implication in Proposition~\ref{prop:CS} does
not hold (Example~\ref{ex:caviglia-counter}).  In fact, it would imply
Stanley's conjecture for cover ideals of simple matroids.  To see
this, let $I\subset S$ be an ideal such that $S/\gin(I)$ is level.
In characteristic zero, the
generic initial ideal is strongly stable and thus $x_n,x_{n-1},\ldots
,x_{d+1}$ is a regular sequence in~$S/\gin(I)$.  The Artinian
reduction $S/(\gin(I)+(x_n,x_{n-1},\ldots ,x_{d+1}))$ is an Artinian
level monomial algebra with the same $h$-vector as~$S/I$.  In fact,
having a binomial regular sequence would suffice to ensure monomiality
of the quotient (see Remark~\ref{rem:lgin}).  Consequently, the
$h$-vector of $S/I$ is a pure $O$-sequence.  If the converse of
Proposition \ref{prop:CS} were true, then the $h$-vector of any level
algebra whose second to last syzygy module is componentwise linear would be a
pure $O$-sequence.  This is the case for cover ideals of \emph{simple}
matroids, that is matroids without parallel elements:
\begin{proposition}\label{prop:bettiImpliesSyzCompLin}
  Let $\D$ be a rank $d$ simple matroid on $n$ vertices. Then
  $\beta_{d-1,j}(S/J(\D))\neq 0$ only for $j=n-1$.  In particular,
  $Z_{d-2}(S/J(\D))$ is componentwise linear.
\end{proposition}
\begin{proof}
  Let $\Gamma = \Dc$. Hochster's formula implies:
\[
\beta_{d-1,j}(S/J(\D))=\beta_{d-1,j}(S/I_{\Gamma})=\sum_{\substack{W\subset
    [n]\\ |W|=j}} \dim_{\Bbbk} \tilde{H}_{j-d}(\Gamma_W,\Bbbk),
\]
where $\Gamma_{W}$ denotes the restriction of $\Gamma$ to the vertex
subset~$W$.
If $j>n-1$, then the only summand that could occur is $\dim_{\Bbbk}
\tilde{H}_{n-d}(\Gamma,\Bbbk) = 0$ in the case $j=n$.
If $j<n-1$, then we can find two distinct vertices outside
of~$W$. Since $\D$ is simple, they must be contained in a facet $F$ of
$\D$. Therefore, $G:=[n]\setminus F$ is a facet of $\Gamma$, and
$|G\cap ([n]\setminus W)|\leq n-j-2$.  Thus $\dim (\Gamma_{W}) \geq
j-d +1$.  By Reisner's criterion $\widetilde{H}_{j-d}(\Gamma_W,\Bbbk)
= 0$ since $\Gamma_{W}$ is a matroid and can thus have only
top-dimensional homology.
\end{proof}

\begin{example}\label{ex:caviglia-counter}
  Let $\Delta$ be the rank three simple matroid on $\{1,\dots,7\}$ with the
  following facets 
  \begin{gather*}
    123, 124, 125, 127, 135, 136, 137, 145, 146, 147, 156, 167, 234,
    235, \\ 236, 246, 247, 256, 257, 267, 345, 346, 347, 357, 367,
    456, 457, 567,
  \end{gather*}
  commonly known as the Fano matroid.
  A quick computation with Macaulay2 shows that the cover ideal is
  level of Cohen-Macaulay type $8$ while its generic initial ideal is
  not level ($\beta_{3}(S/\gin(J(\Delta)) = 10$).  Since $\Delta$ is
  simple, Proposition~\ref{prop:bettiImpliesSyzCompLin} shows that
  $Z_{1} (S/J(\D))$ is componentwise linear.
\end{example}

\begin{remark}\label{rem:lgin}
  Propositions~\ref{prop:CS} and~\ref{prop:bettiImpliesSyzCompLin}
  inspired the search for a less generic initial ideal in which the
  coordinate transform has block structure.  The hope was to find a
  construction that balances between preserving the last Betti
  number---yielding a \emph{level} quotient---and maintaining the
  existence of a binomial regular sequence---needed to have a
  \emph{monomial} quotient.  However, we did not find a definition 
  that realizes just the right balance.
\end{remark}

If the generic initial ideal of $I_{\Delta}$ is level, then $h_{\D}$
is a pure $O$-sequence since it equals the Hilbert function of the
Artinian reduction of $\gin(I_{\D})$ by variables.  To implement this
strategy we employ the following two general lemmas.  Following
\cite{HH99}, let $I_{<k}$ denote the subideal of a homogeneous ideal
$I$ generated by the homogeneous elements of $I$ of degree less
than~$k$.
\begin{lemma}\label{lem:gin}
  Let $I\subset S$ be a homogeneous ideal of projective dimension~$p$
  and regularity~$k$. If $\pd(I_{<k})<p$ and $\text{char}(\Bbbk) = 0$,
  then
  $\beta_p(I)=\beta_{p,p+k}(I)=\beta_{p,p+k}(\gin(I))=\beta_p(\gin(I))$.
\end{lemma}
\begin{proof}
  Let $J_1=\gin(I)_{<k}$ and $J_2=\gin(I_{<k})_{<k}$.  It is easy to
  see that $J_1=J_2$.  In characteristic zero, the generic initial
  ideal is strongly stable and~\cite[Theorem 2.4(a)]{BS} shows
  $\pd(J_2)<p$. Using the Eliahou-Kervaire
  resolution~\cite[Theorem 2.1]{EK}, we get that no monomial
  $x_{p+1}u$ with $u\in\Bbbk[x_1,\ldots,x_{p+1}]$ is a minimal
  generator of~$J_2=J_1$.  Therefore any minimal generator of
  $\gin(I)$ of the form $x_{p+1}u$ with
  $u\in\Bbbk[x_1,\ldots,x_{p+1}]$ must be of degree at least~$k$.
  Since by \cite[Theorem~2.4(b)]{BS} we have $\reg(I)=\reg(\gin(I))$
  it must be of degree exactly~$k$.

  The Eliahou-Kervaire formula \cite[Corollary~7.2.3]{HH11} gives one
  of the equations: $\beta_p(\gin(I))=\beta_{p,p+k}(\gin(I))$.  Since
  $\beta_{p,p+k}(I)$ is an extremal Betti number, we have
  $\beta_{p,p+k}(\gin(I))=\beta_{p,p+k}(I)$
  by~\cite[Corollary~1.3]{BCP}.  Finally, it is a general fact (see
  for example \cite[Theorem~8.29]{MS05}) that $\beta_{p,p+j}(I)\leq
  \beta_{p,p+j}(\gin(I))$ for any~$j$, so actually
  $\beta_p(I)=\beta_{p,p+k}(I)=\beta_{p,p+k}(\gin(I))=\beta_p(\gin(I))$.
\end{proof}

\begin{lemma}\label{lem:cmaddingfacets}
  Let $\D$ be a Cohen-Macaulay complex of dimension $d$, and $F$ a
  minimal non-face of cardinality $d+1$. Then $\D\cup F$ is
  Cohen-Macaulay.
\end{lemma}
\begin{proof}
  Let $\<F\>$ denote the complex on $[n]$ with one facet~$F$.  By
  construction $\langle F\rangle\cap \D$ is the boundary of a
  $d$-simplex. In particular $\Bbbk[\langle F\rangle\cap\D]$ is a
  $d$-dimensional Cohen-Macaulay ring.  So the statement follows at
  once from the exact sequence
\[
0\rightarrow \Bbbk[\D\cup F]\rightarrow \Bbbk[\D]\oplus \Bbbk[\langle
F\rangle]\rightarrow \Bbbk[\langle F\rangle\cap\D]\rightarrow 0,
\]
and the depth inequalities.
\end{proof}

The following theorem is the main result of this section.  We state it
for Stanley-Reisner ideals.

\begin{theorem}\label{thm:skeleton}
  Let $\D$ be the $(d-1)$-skeleton of a $d$-dimensional Cohen-Macaulay
  complex. Then $h_{\Delta}$ is a pure $O$-sequence.  Furthermore, if
  $\text{char}(\kk) = 0$, then $\kk[\D]$ is level.
\end{theorem}
\begin{proof}
  By Hochster's formula $\reg(\Bbbk[\D]) \leq d$ and since
  $I_{\Delta}$ has a generator of degree $d+1$, \mbox{$\reg(\Bbbk[\D])
    = d$}.  Write $J=(I_{\D})_{< d+1}$ and let $\Gamma$ be the
  corresponding simplicial complex.  The
  result follows from Lemma~\ref{lem:gin} once we show
  $\depth(\Bbbk[\Gamma])>d$, which, in turn, is equivalent to the
  $d$-skeleton of $\Gamma$ being Cohen-Macaulay.  The $d$-skeleton of
  $\Gamma$ is the complex $\Gamma_{d}$ that arises from $\Delta$ by
  turning all non-faces of size $d+1$ into facets.  Now, $\D$ is the
  $(d-1)$-skeleton of a $d$-dimensional Cohen-Macaulay complex
  $\Omega$.  There are two kinds of facets of $\Gamma_{d}$: those that
  are facets of $\Omega$ and those that are not.  Those that are not,
  are minimal non-faces in $\Omega$.  By
  Lemma~\ref{lem:cmaddingfacets}, $\Gamma_{d}$ is Cohen-Macaulay.  The
  statement about the $h$-vector is characteristic-free because the
  $h$-vector of a simplicial complex does not depend on the
  coefficient field.
\end{proof}

It is equivalent to say that a vector is the Hilbert function of an
Artinian monomial algebra and that it is the $f$-vector of an
\emph{order ideal of monomials}, also known as a \emph{multicomplex}.
In this language pure $O$-sequences are $f$-vectors of pure
multicomplexes.  Similar to simplicial complexes, there are theories
of shellability of multicomplexes (such as M-shellability) and the work of Chari suggests that
a characterization of $f$-vectors of shellable multicomplexes may be
possible~\cite{CH97}.  He also conjectures that the $h$-vector of any
coloop-free matroid is a \emph{shellable
  $O$-sequence}~\cite[Conjecture~3]{CH97} which would imply Stanley's
conjecture.
\begin{remark}\label{rem:shellable}
  Let $I\subset S = \kk[x_{1},\dots,x_{r}]$ be a strongly stable ideal
  such that $S/I$ is an Artinian level ring. In this case the
  $h$-vector of $S/I$ is the $f$-vector of an M-shellable
  multicomplex.
\end{remark}
\begin{proof}
  By the Eliahou-Kervaire resolution, the variable $x_r$ appears only
  in the minimal generators of $I$ of maximal degree.  Let $k$ be this
  maximal degree, and let $u_1,\ldots ,u_t$ be the degree $k$ minimal
  generators of $I$ divisible by $x_r$.  Write $u_i=v_ix_r$ for all
  $i=1,\ldots ,t$.  One easily checks that $v_1,\ldots ,v_t$ generate
  the order ideal of $S/I$.  Let $\prec$ be the graded revlex order
  induced by $x_r>\ldots >x_1$.  We can assume $v_1\prec\ldots \prec
  v_t$.  Now write $v_i=v_i'x_r^{e_i}$, where $e_i$ is the maximum
  power of $x_r$ dividing $v_i$, and let $V_i=\{v_i'x_r^j:j=0,\ldots
  e_i\}$.  We claim that $V_t,\ldots ,V_1$ is a shelling of the
  multicomplex $S/I$.  It remains to show that, if $u$ is a monomial
  of degree $e$ dividing $v_i$, then there exists $j\geq i$ such that
  $v_j=ux_r^{k-e-1}$.  Let $m$ be the monomial of degree $k-e-1$ such
  that $v_i=um$.  If no such $j$ existed, then $ux_r^{k-e-1}$ would be
  in $I$, so there would exist a minimal generator $u'$ of $I$, say of
  degree $a$, such that $u=u'u''$ for some~$u''$.  Then $u'x_r^{e-a}$
  would be in $I$ as well.  Since $I$ is strongly stable,
  $u=u'x_r^{e-a}/x_r^{e-a}\cdot u''\in I$.  This is a contradiction to
  $u_i$ being a minimal generator.
\end{proof}

In matroid theory, passing from a matroid of rank $d$ to its
$k$-skeleton for $k<d-1$ is called a \emph{truncation}.  The rank
function of the truncation is $A \mapsto \min\{\rk(A),k+1\}$.  The
shift of one arises because the $k$-skeleton is of dimension~$k$ which
means rank~$k+1$.  All together we have

\begin{theorem}\label{thm:StanleyForTruncation}
  Any truncation of a matroid satisfies Chari's conjecture and
  consequently also Stanley's conjecture.
\end{theorem}
\begin{proof}
  If $I\subset S$ is a strongly stable ideal such that $S/I$ is level,
  then the $h$-vector of $S/I$ is the $f$-vector of an M-shellable
  multicomplex by Remark~\ref{rem:shellable}.  By
  Theorem~\ref{thm:skeleton}, the $h$-vectors of truncated matroids
  satisfy Chari's and consequently also Stanley's conjecture.
\end{proof}

Evidently the next question is: Which matroids are truncations?
Certainly not all of them.
\begin{example}
  Any complete bipartite graph is a rank two matroid that is not the
  truncation of a matroid.  More generally, any matroid that becomes a
  simplex after identifying parallel elements is not a truncation.
\end{example}

\begin{remark}
  If a rank $d$ matroid $\Gamma$ is a truncation, then it is a
  truncation of a rank $d+1$ matroid $\Delta$.  In this case, any
  facet of $\Delta$ is a \emph{spanning circuit} of $\Gamma$, that is,
  a minimal non-face of size~$d+1$.  In particular, the facets of
  $\Delta$ are contained in the spanning circuits of~$\Gamma$.
  Moreover, if $\Gamma$ has no spanning circuit, then it is not the
  truncation of a matroid.
\end{remark}

\begin{example}
  The dual of the Fano matroid from Example~\ref{ex:caviglia-counter}
  has no spanning circuit.
\end{example}

\begin{remark} Let $\Delta$ be a matroid which has a spanning circuit.
  In {\cite{brylawski86:_const}} Brylawski gives an algorithm that
  decides if there exists a matroid $\Gamma$ such that $\Delta$ is the
  truncation of $\Gamma$, and constructs the \emph{freest} such
  matroid whenever possible.
\end{remark}

In the remainder of the section we discuss Schubert matroids (also
known as shifted matroids, PI-matroids, and generalized Catalan
matroids~\cite{fink2010matroid}).  They play an important role in the
study of Hopf
algebras of (poly)matroids~\cite{derksen2010valuative}.

\begin{definition}
  Let $1\leq s_1<s_2<\ldots <s_d\leq n$ be a sequence of strictly
  ascending integers.  The \emph{Schubert matroid} $SM_n(s_1,\ldots,
  s_d)$ is the rank $d$ matroid on $[n]$ with facets:
  \begin{equation}\label{eq:schubertCondition}
    \left\lbrace \{i_1,\ldots ,i_d\} \ \ : \ \ i_j\leq s_j
    \right\rbrace.
  \end{equation}
\end{definition}

\begin{remark}\label{rem:truncation}
  For any simplicial complex $\Delta$, the ideal $(I_{\D})_{\<k\>}$,
  generated by the degree $k$ part of $I_{\D}$ is generated by all
  monomials corresponding to non-faces of size~$k$.
\end{remark}

\begin{lemma}\label{lem:truncate-schubert}
  If $\Delta = SM_{n}(s_{1},\dots,s_{d})$ is a Schubert matroid of
  rank $d$ and $s_{1}\geq 2$, then for any $k<d+1$, $(I_{\D})_{\<k\>}$
  is the ideal generated by the degree $k$ part of the Stanley-Reisner
  ideal of $SM_{n}(s_{1}-1,s_{1},\dots,s_{d})$.
\end{lemma}
\begin{proof}
  If $\{j_{1},\dots,j_{d+1}\}$ is a facet of
  $SM_{n}(s_{1}-1,s_{1},\dots,s_{d})$ then it is a minimal non-face of
  $SM_{n}(s_{1},\dots,s_{d})$ since any 
  $\{j_{1},\dots,\widehat{j_{l}},\dots,j_{d+1}\}$
  satisfies~\eqref{eq:schubertCondition}.
  On the other hand, if $\{j_{1},\dots,j_{d+1}\}$ is a non-face of
  $SM_{n}(s_{1}-1,s_{1},\dots,s_{d})$, then $\{j_{2},\dots,j_{d+1}\}$ is
  a non-face of $SM_{n}(s_{1},\dots,s_{d})$, assuming without loss of
  generality that $j_1<j_2<\ldots <j_{d+1}$.
  By Remark~\ref{rem:truncation} the statement holds for any $k<d+1$.
\end{proof}

\begin{theorem}\label{thm:schubert}
  Schubert matroids have componentwise linear Stanley-Reisner ideals
  and in particular satisfy Chari's (and thus Stanley's) conjecture.
\end{theorem}
\begin{proof} 
  If $s_1=1$, then $SM_n(s_1,\ldots ,s_d) \cong SM_{n-1}(s_1-1,\ldots
  ,s_d-1)*\{v\}$.  The Stanley-Reisner ideal of $SM_{n-1}(s_1-1,\ldots
  ,s_d-1)*\{v\}$ does not use the variable of $v$ and is componentwise
  linear if and only if the Stanley-Reisner ideal of $SM_n(s_1,\ldots
  ,s_d)$ is componentwise linear.  If $s_d<n$, then $SM_n(s_1,\ldots
  ,s_d)\cong SM_{s_d}(s_1,\ldots ,s_d)$.  The Stanley-Reisner ideal of
  $SM_{s_d}(s_1,\ldots ,s_d)$ equals that of $SM_n(s_1,\ldots ,s_d)$
  plus variables.  One is componentwise linear if and only the
  other~is. Consequently, assume $1<s_1<s_2<\ldots <s_d=n$.
  We proceed by induction on the corank $n-d$.  The base case is
  $n-d=1$ in which $I_{\D}$ is principal.  
  To check that $I$ is componentwise linear, it suffices to check
  $I_{\<k\>}$ for any $k$ in which $I$ has minimal
  generators~\cite{HH99}, and $I_{\D}$ has minimal generators in
  degrees $\le d+1$.
  Since $\reg(I_{\D}) = d+1$, the ideal $(I_{\Delta})_{\<d+1\>}$ has a
  linear resolution~\cite[Proposition~1.1]{EG84}.  By
  Lemma~\ref{lem:truncate-schubert} and the induction hypothesis we
  conclude.
\end{proof}

\section{\texorpdfstring{Matroids with $d+2$ parallel classes}
  {Matroids with d+2 parallel classes}}
\label{sec:partite}

In the remainder of the paper we focus on duals of matroids, or
equivalently, $h$-vectors of cover ideals.  If $\Delta$ is a matroid,
then $h^\D = h_{\Dc}$ is the $h$-vector of $S/J(\Delta)$, the
quotient by the cover ideal of~$\Delta$.  
The one-dimensional skeleton of a matroid is a complete $p$-partite
graph whose groups of vertices correspond to the partition of the
vertex set of the matroid set into \emph{parallel
  classes}~\cite[Corollary~2.3]{CV3}. The main result of this section (Theorem~\ref{thm:d2}) says that
Stanley's conjecture holds for cover ideals of matroids whose number
of parallel classes is at most two more than the rank.   Due to the technical nature of the proof, we divide it into several smaller results, give various examples along the way, and state the general theorem  at the very end.  

Our notation follows closely that of~\cite{CV3}.  Let $\Delta$ be a matroid of
rank~$d$, with parallel classes $A_1,\ldots,A_p$, of cardinalities
$a_1,\ldots, a_p$.  Such matroids are \emph{$p$-partite}.  The
\emph{simplification} $\Ds$ of $\Delta$ is the matroid that arises
from $\Delta$ by replacing each parallel class by a single vertex.  We
begin with a technical condition to be used in many inductive
constructions.
 
\begin{lemma}\label{lem:compatibility}
  Let $\Gamma ' = \langle N_1,\ldots,N_u\rangle$ be a pure order ideal
  in variables $y_1,\ldots,y_{d}$, and let $\Gamma'' = \langle
  M_1,\ldots,M_{v}\rangle$ be a pure order ideal in the variables
  $y_1,\ldots,\widehat{y_{r}},\ldots,y_{d}$, that is, not using
  $y_{r}$.  Assume that $h^{\Delta\setminus A_p} = f(\Gamma')$ and
  that $h^{\lk_{\Delta}A_p} = f(\Gamma'')$.  Suppose that
  $\forall~i\in[u]$, $\exists~ j\in [v]$ such that
\begin{equation}\label{eq:comp}
  \frac{N_i}{y_r^{n_i}} \mid M_j,\quad\text{where } n_i = \max\{m:y_r^m \mid N_i\}.
\end{equation}
Then $h^{\Delta}$ equals the $f$-vector of the pure order ideal
\[
\Gamma = \langle y_r^{a_p}N_1,\ldots,y_r^{a_p}N_u,
y_r^{a_p-1}M_1,\ldots,y_r^{a_p-1}M_v\rangle.
\]
\end{lemma}
\begin{proof} 
  By \cite{CV3} we have for any $i\ge 0$ that
  \[h^{\Delta}_i = h^{\Delta\setminus A_p}_{i-a_p} +
  \sum_{j=0}^{a_p-1} h_{i-j}^{\lk_{\Delta}A_p}.\] It suffices to
  show the corresponding formula for $\Gamma$:
  \[
  f_i(\Gamma) = f_{i-a_p}(\Gamma') + \sum_{j=0}^{a_p-1}
  f_{i-j}(\Gamma'').
  \]
  Fix an index $i$ and write $\Gamma_i = \{M \in \Gamma:\deg M =
  i\}$.  We write $\G_i$ as the disjoint union $ G_{\ge a_p}\sqcup
  G_{a_p-1} \sqcup \ldots \sqcup G_0$, where $G_j = \{M \in
  \Gamma_i:y_r^j\mid M \textup{ but } y_r^{j+1}\nmid M\}$,
  and $G_{\ge a_p}=\{M\in \Gamma_i:y_r^{a_p}\mid M\}.$
  If a generator of $\Gamma$ is divisible by $y_r^{a_p}$, then it cannot come from generators of $\Gamma''$.  Hence $ f_{i-a_p}(\Gamma')
  = |G_{\ge a_p}|$, and it suffices to check that $f_{i-j}(\Gamma'') =
  |G_{a_p-j-1}|$.
  The inequality $f_{i-j}(\Gamma'') \le |G_{a_p-j-1}|$ follows from
  the definition of $\Gamma$. To obtain equality we confirm that each
  monomial in $G_{a_p-j-1}$ divides some generator
  $y_r^{a_p-1}M_l$. Assume there exists a monomial $M =
  y_r^{a_p-j-1}M' \in \Gamma$ (with $y_r\nmid M'$), such that $M\mid
  y_r^{a_p}N_k$, for some $k$.  By \eqref{eq:comp}, there exists $l$ such
  that
  \[ 
  \frac{N_k}{y_r^{n_k}} \mid M_l.
  \] 
  This implies that $M' \mid M_l$, and as $a_p-j-1\le a_p-1$ we
  conclude.
\end{proof}

In our inductive proofs, the matroids $\G'$ are special simplicial
complexes for which Stanley's conjecture is known by \cite{CV3}.  They
are defined as follows.
\begin{definition}\label{def:DeltaT}
  Let $\aaa = (a_1,\dots,a_p)$ be a vector of positive integers.  Fix
  integers $2\leq d\le p$ and $0\leq t \leq d-2$. Let $A_1,\dots, A_p$
  be disjoint sets of vertices with $|A_i|=a_i$ for any $i$. The
  matroid $\Delta_{t}(d,p,\aaa)$ is the rank $d$ matroid on
  $\sum_{i}a_{i}$ vertices with facets
  \begin{equation*}
    A_{i_{1}}\dots A_{i_{d-t}} A_{p-t+1} \dots A_p  \quad \text {where }
    1\leq i_{1} < \dots < i_{d-t} \leq p-t.
  \end{equation*}
  Here $A_{j_1}\dots A_{j_k}$ stands for all sets
  $\{v_{j_1},\dots,v_{j_k}\}$ such that $v_{j_i}\in A_{j_i}$.  The
  matroid $\Delta_{0}(d,p,\aaa)$ is the \emph{complete matroid} of
  rank $d$ with $p$ parallel classes of sizes $a_{1},\dots,a_{p}$.
\end{definition}
The simplification of $\Delta_{t}(d,p,\aaa)$ is isomorphic to
$\Delta_{t}(d,p,\mathbf{1})$, which in turn equals the simplicial join
of the uniform matroid $U_{d-t,p-t}$ of rank $d-t$ on $p-t$ vertices,
with a simplex on $t$ vertices.  The matroids $\Delta_{t}$ appear in
\cite{CV3} with a different numbering of the parallel classes, but
here we find this convention more natural.  The $h$-vector of the
cover ideal of $\Delta_{t} (d,p,\aaa)$ is a pure $O$-sequence by
\cite[Theorem~3.7]{CV3} and we give its order ideal in
Example~\ref{ex:orderIdealDt}, after setting up a useful notation.
\begin{notation}
  Fix positive integers $(a_1,\ldots,a_p)$.  For any set partition
  $\mathcal{P} = P_{1} \sqcup \dots \sqcup P_{d}$ of $[p]$, denote by
  $[\mathcal{P}]= [ P_{1} | P_{2} | \dots | P_{d}]$ the monomial in
  $d$ variables:
  \[ y_1^{-1+\sum_{j\in P_{1}}a_{j}} \cdot\ldots\cdot
  y_{d}^{-1+\sum_{j\in P_{d}}a_{j}}.\] When no confusion may arise, we
  will use this notation for the corresponding partition as well.
\end{notation}

\begin{example}\label{ex:orderIdealDt}
  Fix integers $t, d, p$ such that $0\le t\le d -2\le p-2$, and an
  integer vector $\aaa = (a_{1},\dots, a_{p})$.  For any ascending
  sequence $1 = l_{0} < l_{1} < \dots < l_{d} = p+1$ of integers, let
  $\mathcal{P}(l_{0},\dots,l_{d})$ be the $d$-partition into sets
  $P_{i} = \{l_{i-1},\dots,l_{i}-1\}$.  We define the following pure
  order ideal:
  \[
  \G_{t}(d,p,\aaa) := 
  \<[\mathcal{P}(l_{0},\dots,l_{d-t})|p-t+1|\dots|p] : \text{ for all } 
  1 = l_{0} < l_{1}
  < \dots < l_{d-t} = p-t+1 \>.
  \]
  In particular, when $t=0$ we have
  \[
  \G_{0}(d,p,\aaa) := \< [\mathcal{P}(l_{0},\dots,l_{d})] : \text{ for
    all } 1 = l_{0} < l_{1} < \dots < l_{d} = p+1 \>.
  \]
  By \cite[Theorem~3.7]{CV3} the vector $h^{\D_{t}(d,p,\aaa)}$ equals
  the $f$-vector of $\G_t(d,p,\aaa)$. This equality is not easy to
  check in general. One may prove it by induction for complete
  matroids, then notice that $$\D_t(d,p,(a_1,\dots,a_p)) =
  \D_0(d-t,p-t,(a_1,\dots,a_{d-t})) \ast
  \D_0(t,t,(a_{p-t+1},\dots,a_p)),$$ check that a similar equality
  holds for the pure order ideals (viewed as multicomplexes), and
  finally use the behavior of $h$-vectors and $f$-vectors over star
  products.
%
In this section we are mainly interested in the case $p=d+1$, where
  $\G_{t}(d,d+1,\aaa)$ is generated by
  \small
\[ 
\begin{array}{rcccccccccccccccl}
  \lbrack&1&|&2&|&\cdots&|&d-t+1&|&d-t,d-t+1&|&d-t+2&|&\cdots&|&d+1&\rbrack\\
  \lbrack&1&|&2&|&\cdots&|&d-t+1,d-t&|&d-t+1&|&d-t+2&|&\cdots&|&d+1&\rbrack\\

    &&&&&&&& \vdots &&&&&&&&\\
  \lbrack&1&|&2,3&|&\cdots&|&d-t&|&d-t+1&|&d-t+2&|&\cdots&|&d+1&\rbrack\\
    \lbrack&1,2&|&3&|&\cdots&|&d-t&|&d-t+1&|&d-t+2&|&\cdots&|&d+1&\rbrack.
  \end{array}\]
\normalsize
In particular, for $t=1, d=3, p=4$ and some $\aaa$ we obtain:
\[\begin{array}{rcl}
\G_{1}(3,4,\aaa) &=&   \<[\mathcal{P}(l_{0},l_1,l_{2})~|~ 4] : \textup{ for all }   1 = l_{0} < l_{1}< l_{2} = 4 \> \\
\rule{0pt}{4ex}&=&  \<[\mathcal{P}(1,2,4) ~|~ 4], [\mathcal{P}(1,3,4) ~|~ 4]\> \\
\rule{0pt}{4ex}&=&   \<[1 ~|~ 2,3 ~|~ 4], [1,2 ~|~ 3 ~|~ 4]\>\\
\rule{0pt}{4ex}&=& \< y_1^{a_1-1}y_2^{a_2+a_3-1}y_3^{a_4-1},~~ y_1^{a_1+a_2-1}y_2^{a_3-1}y_3^{a_4-1}\>.\\
\end{array}\]
Plugging in various values for $\aaa$ one can directly check $h^{\D_{1}(3,4,\aaa)} = f_{\G_{1}(3,4,\aaa)}$.
\end{example}

\begin{definition}
  Let $[P_{1}|\cdots|P_{d}]$, $[Q_{1}|\cdots|Q_{d}]$ be $d$-partitions
  of subsets of $[p]$.  For every vector of positive integers $\aaa =
  (a_1,\dots,a_p)$, let $\leq_{\aaa}$ be the partial  order 
  defined by
  \[
  [P_{1}|\cdots|P_{d}] \leq_{\aaa} [Q_{1}|\cdots|Q_{d}] \quad 
  \Longleftrightarrow \quad \sum_{j\in P_{i}} a_{j} \leq \sum_{j\in Q_{i}} a_{j},
  \text{ for all $i = 1,\dots,d$.}
  \]
  For any $(d-1)$-partition $[Q'_1|\dots|Q'_{d-1}]$ of $[p]$ and integer
  $r\in[d]$, a partial order $\leq_{\aaa}^{r}$ is defined by
  \[[P_{1}|\cdots|P_{d}] \leq_{\aaa}^r [Q'_{1}|\cdots|Q'_{d-1}] \quad
  \Longleftrightarrow \quad [P_1|\dots|\widehat{P_r}|\dots|P_d]
  \leq_{\aaa}[Q'_{1}|\cdots|Q'_{d-1}].  \]
\end{definition}

The compatibility condition~\eqref{eq:comp} in
Lemma~\ref{lem:compatibility} can be rewritten using the new notation.
\begin{definition}\label{def:partition comp}
  Let $\mathrm{P} = \{\mathcal{P}_{1}, \dots \mathcal{P}_{s}\}$ be a
  set of $d$-partitions of $[p]$, $\mathrm{Q} =
  \{\mathcal{Q}_{1},\dots,\mathcal{Q}_{r}\}$ a set of
  $(d-1)$-partitions of $[p]$. For every $r\in[d]$ we say that the
  sets $\mathrm{P},\mathrm{Q}$ satisfy the \emph{$r$-compatibility
    condition} if for each $\mathcal{P} \in\mathrm{P}$ there exists a
  $\mathcal{Q}\in\mathrm{Q}$ such that $\mathcal{P} \leq_{\aaa}^r
  \mathcal{Q}$.
\end{definition} 

\begin{example}\label{ex:favorite}
  The sets of partitions $\mathrm{P} =
  \{[1|2|3,4],[1|2,3|4],[1,2|3|4]\}$ and $\mathrm{Q}=\{[1,2|3,4]\}$
  are 3-compatible if and only if $a_2 \le a_4$, while the collections
  $\mathrm{P'} = \{[1|2,3|4,5],[1,2|3|4,5]\}$ and
  $\mathrm{Q'}=\{[1|2,3,4,5], [1,2|3,4,5], [1,2,3|4,5],[1,2,3,4|5]\}$
  are $i$-compatible for any $\aaa$ and any $i = 1,2,3$.
\end{example}
In the new notation, the \emph{gluing} in
Lemma~\ref{lem:compatibility} takes two sets $\G'$ and $\G''$ of
partitions of~$[p-1]$ and produces a set of $d$-partitions of $[p]$.
The procedure consists of
\begin{itemize}
\item[-] adding the element $p$ to each $r$th set of a partition
  in~$\G'$,
\item[-] inserting the set $\{p\}$ into each partition of $\G''$ as
  the $r$th set, shifting the index of the last $d-r$ sets by one.
\end{itemize}
Here is an example of how Lemma \ref{lem:compatibility} can be
applied.  It is one of the base cases in the proof of
Proposition~\ref{prop:d+2partite}.

\begin{example}\label{ex:alex-method}
  Let $\D$ be the rank 3 matroid with 5 parallel classes and facets:
  \[ 
  A_1A_2A_3,\, A_1A_2A_4,\, A_1A_3A_4,\, A_2A_3A_4,\, A_1A_3A_5,\,
  A_1A_4A_5,\, A_2A_3A_5,\, A_2A_4A_5.
  \] 
  As $\D\setminus A_5 = \D_0(3,4,(a_1,a_2,a_3,a_4))$, it holds that
  $h^{\D\setminus A_5} = f(\G_0(3,4,(a_1,a_2,a_3,a_4))$, corresponding
  to
  \[ 
  \mathrm{P} = \{[1|2|3,4],[1|2,3|,4],[1,2|3|4]\}.
  \] 
  The rank 2 matroid $\lk_\D A_5$ is the complete bipartite graph
  $\D_0(2,2,(a_1+a_2,a_3+a_4))$, and thus its $h$-vector is obtained
  from the order ideal generated by $\mathrm{Q}=\{[1,2|3,4]\}$.
  Example \ref{ex:favorite} shows that P and Q are 3-compatible if and
  only if $a_2\le a_4$.  Switching the pairs $(A_1,A_2)$ and
  $(A_3,A_4)$ in $\D$ gives an isomorphic matroid, therefore we may
  assume without loss of generality that $a_2\le a_4$, and obtain by
  Lemma~\ref{lem:compatibility} that $h^\D=f(\G)$ for
   \[\G = \langle [1|2|3,4,5],[1|2,3|4,5],[1,2|3|4,5],[1,2|3,4|5] \rangle.\]
\end{example}
A crucial property of $(d+2)$-partite matroids is that they possess a
dual graph, which together with the vector $(a_1,\dots,a_p) $
completely encodes their structure.
\begin{definition} Let $\Delta$ be a matroid of rank $d$ with $d+2$
  parallel classes and let $\Ds$ be its simplification.  The graph
  $G_{\Delta}$ is the rank two matroid~$(\Ds)^{\mathrm{c}}$.
\end{definition} 
By construction $G_\D$ is a complete $q$-partite graph on $[d+2]$, for
some $q\in\{2,\dots,d+2\}$.  If $G_{\D}$ is a complete graph on $d+2$
vertices (i.e. if $q=d+2$), then its dual is the complete
$(d+2)$-partite matroid, for which Stanley's conjecture holds by
\cite[Theorem~3.5]{CV3}.  However, not all complete $q$-partite graphs
have simple matroids as their duals.
\begin{remark}\label{rem:excluded-graphs}
  For every $d\ge 2$, the bipartite graph with partition $\{1,2\} \cup
  \{3,\ldots,d+2\}$ and the tripartite graph with partition $ \{1\} \cup
  \{2\} \cup \{3,\ldots,d+2\}$ have duals in which  $1$ and $2$ are parallel
  and these are the only $n$-partite graphs with this property.
\end{remark}
\begin{proof}
  The set $\{1,2\}$ is a minimal non-face in the dual of a complete
  $n$-partite graph $G$ if and only if every edge of $G$ has at least
  one of $1$ and $2$ as a vertex.
\end{proof}
\begin{remark}\label{rem:G<->Delta}
  The $[(d-1)+2]$-partite matroid $\lk_\D A_i$ of rank $d-1$
  corresponds to the deletion of $i$ in $G_\D$, that is $G_{\lk_\D
    A_i}=G_\D \setminus i$.
  The $(d+1)$-partite matroid $\D\setminus A_i$ of rank $d$
  corresponds to $\lk_{G_\D}i$ viewed as a matroid on
  $[d+2]\setminus\{i\}$. That is, if $j$ is parallel to $i$ in~$G_\D$,
  it is a loop in the rank one matroid $\lk_{G_\D}i$. 
%
  If the parallel class in~$G_\D$  of  $d+2$ (the vertex corresponding to the 
  parallel class $A_{d+2}$ in $\Delta$) has cardinality $s$, then
  \[ 
  \D\setminus A_{d+2} \cong \D_{s-1}(d,d+1,(a_1,\dots,\,a_{d+1})).\]
  Similar isomorphisms hold for the deletions of the other parallel
  class $A_i$ and each one is determined by which vertices of $G_\D$
  are parallel to~$i$.
\end{remark}
Our proof of Theorem \ref{thm:d2} is an induction on the number of
vertices of~$G_\D$.  Remark~\ref{rem:excluded-graphs} implies that
there are three different bases of induction to consider, dividing the
proof into three cases:
\begin{enumerate}
\item $G_\Delta$ has at most one parallel class of cardinality $\ge 2$.
\item $G_\Delta$ is bipartite.
\item $G_\Delta$ is $r$-partite for $r\ge3$, and at least two parallel classes of cardinality $\ge 2$.
\end{enumerate}
\begin{proposition}\label{prop:case1}
  If $G_\D$ is complete $n$-partite on
  $\{1,\ldots,r\}\cup\{r+1\}\cup\ldots\cup\{d+2\}$, for some $r\geq
  1$, then $h^\D$ is a pure $O$-sequence.
\end{proposition}
\begin{proof}
  The proof is by induction on the number $d+2-r$ of singleton
  classes.  By Remark~\ref{rem:excluded-graphs}, the base case is
  $d+2-r=3$, since for larger $r$ the graph $G_{\D}$ is not the dual
  of a simple matroid.  Decompose $\Delta$ into deletion and link
  at~$A_{d+2}$.  By Remark~\ref{rem:G<->Delta}, it holds that
  $\D\setminus A_{d+2} = \D_0(d,d+1,(a_1,\dots,a_{d+1}))$, thus its
  $h$-vector is realized by $\G'=\G_0(d,d+1,(a_1,\dots,a_{d+1}))$,
  which is generated by
  \[
\mathrm{P}=\{
  [1|2|\ldots|d-1|d,d+1],
  [1|2|\ldots|d-1,d|d+1],\dots,[1,2|3|\ldots|d|d+1]\}.\]
  By Remark~\ref{rem:excluded-graphs}, $A_{d}$ and $A_{d+1}$ are
  parallel in $\lk_\D A_{d+2}$, so by Remark~\ref{rem:G<->Delta} we have that $\lk_\D
  A_{d+2}$ is the matroid $
  \D_0(d-1,d,(a_1,\dots,a_{d-1},a_{d}+a_{d+1}))$.  Thus $h^{\lk_\D
    A_{d+2}}=f(\G'')$, with $\G''$ generated by
  \[ \mathrm{Q} =
  \{[1|2|\ldots|d-1,d,d+1],[1|2|\ldots|d,d+1],\ldots,[1,2|3|\ldots|d,d+1]\}.  \]
  It is easy to check that P and Q are $d$-compatible.
  
  In the induction step $\G'$ is as above and $\G''$ is given by the
  inductive hypothesis.  That is to say, we may assume that we applied
  Lemma~\ref{lem:compatibility} $(d-r-1)$ times already, and thus,
  from the last application we have that
  \[
  \G'' \supseteq \langle
  [1|2|\ldots|d-1,d,d+1],[1|2|\ldots|d,d+1],\ldots,[1,2|3|\ldots|d,d+1]\rangle.
  \]
  Compatibility is again straightforward and we conclude.
\end{proof}

The second case, when $G_{\D}$ is bipartite, follows from a general
fact about the join of simplicial complexes (or multicomplexes).  Let
$\Delta$ and $\Delta'$ be two simplicial (multi)complexes on disjoint
vertex sets.  Their \emph{join} is the (multi)complex
$\Delta\ast\Delta' = \{\sigma \cup \sigma'~:~\sigma\in\Delta \textup{
  and } \sigma'\in\Delta'\}$.  The join operation commutes with duals:
$(\Delta\ast\Delta')^{\mathrm{c}} = \Dc\ast\Delta'^{\mathrm{c}}$.  The tensor product
of the Stanley-Reisner rings is the Stanley-Reisner ring of their
join, and by duality, the same statement holds for tensor product of
the quotients by their cover ideals.  In the following remark, the
simplicial join of two order ideals is computed by viewing them as
multicomplexes.

\begin{remark}\label{join} Let $\Delta$ and $\Delta'$ be two matroids, and let $\Gamma$ and $\Gamma'$ be two order ideals.
  If $h^\Delta = f(\Gamma)$ and $h^{\Delta'}=f(\Gamma')$, then
  $h^{\Delta\ast\Delta'}=f(\Gamma\ast\Gamma')$.
\end{remark}
In the next proposition we allow also bipartite graphs with partitions
of cardinality two (i.e. $\D$ is $(d+1)$ partite).  This turns out
useful in the third case.
\begin{proposition}\label{prop:bipartite}
  If $G_\D$ is bipartite with partition
  $\{1,\ldots,s\}\cup\{s+1,\ldots,d+2\}$, then the $h$-vector of the
  cover ideal of $\Delta$ is a pure $O$-sequence.
\end{proposition}
\begin{proof}
  From the bipartition of $G_\D$ we obtain
$$\Delta  = \D_0(s-1,s,\aaa')\ast\D_0(d+1-s,d+2-s,\aaa''),$$
where $\aaa'=(a_1,\ldots,a_s)$ and
$\aaa''=(a_{s+1},\ldots,a_{d+2})$. Thus \cite[Theorem 3.5]{CV3} and
Remark~\ref{join} show that $\Delta$ satisfies Stanley's conjecture.
\end{proof}
\begin{example}\label{ex:explicit}
  If $h^\D= f(\G_0(s-1,s,\aaa')\ast\G_0(d+1-s,d+2-s,\aaa''))$, then an
  explicit description of the order ideal generators follows from
  Example~\ref{ex:orderIdealDt}:
\small
\[
{
\begin{array}{rcccccccccccccccl} 
\lbrack&1&|&\ldots&|&s-2&|&s-1,s&|&s+1&|&\ldots&|&d&|&d+1,d+2&\rbrack\\
\lbrack&1&|&\ldots&|&s-2&|&s-1,s&|&s+1&|&\ldots&|&d,d+1&|&d+2&\rbrack\\
&&&&&&&&\vdots&&&&&&&&\\
\lbrack&1&|&\vdots&|&s-2&|&s-1,s&|&s+1,s+2&|&\ldots&|&d+1&|&d+2&\rbrack\\
\lbrack&1&|&\ldots&|&s-2,s-1&|&s&|&s+1&|&\ldots&|&d&|&d+1,d+2&\rbrack\\
&&&&&&&&\vdots&&&&&&&&\\
&&&&&&&&\vdots&&&&&&&&\\
\lbrack&1,2&|&\ldots&|&s-1&|&s&|&s+1,s+2&|&\ldots&|&d+1&|&d+2&\rbrack.
 \end{array}}
\]
\end{example}
\begin{lemma}\label{lem:induction basis}
  If $G_{\D}$ is tripartite, with partition
  $\{1,\ldots,s\}\cup\{s+1,\ldots,d+1\}\cup\{d+2\}$, where $s\ge 2$
  and $d\ge 4$, then $h^\D$ is a pure $O$-sequence. It equals $f(\G)$,
  where $\G$ is the pure order ideal obtained by applying
  Lemma~\ref{lem:compatibility} to 
  \begin{align*}
    \G' &= \G_0(d,d+1,(a_1,\ldots,a_{d+1})), \quad \text{and}\\
    \G'' &=
    \G_0(s-1,s,(a_1,\ldots,a_s)))\ast\G_0(d+1-s,d+2-s,(a_{s+1},\ldots,a_{d+2})).
  \end{align*}
\end{lemma}
\begin{proof}
  Without loss of generality, assume $a_{s} \le a_{s+1}\le
  \ldots,a_{d+1}$. The matroid $\Delta\setminus A_{d+2}$ equals
  $\D_0(d,d+1,(a_1,\ldots,a_{d+1}))$, so
  $\G'=\G_0(d,d+1,(a_1,\ldots,a_{d+1}))$.  The matroid
  $\lk_{\D_0}A_{d+2}$ corresponds to the bipartite graph from
  Proposition~\ref{prop:bipartite}, thus $\G''$ can be chosen as in
  the statement and Example~\ref{ex:explicit}. To apply
  Lemma~\ref{lem:compatibility}, we check $d$-compatibility of the
  generators of~$\G'$, and~$\G''$.
  Let $P=[1|\dots| i,i+1| \dots| d |d+1]$ be a generator of $\G'$.
  \begin{itemize}[leftmargin=3ex]
  \item If $i\le s-1$, then choose $Q=[1|\dots|i,i+1|\dots| s| s+1|
    \dots| d,d+1]$ and $P\le_{\aaa}^dQ$ for any~$\aaa$.
  \item If $s\le i \le d$, then choose $Q=
    [1|\dots|s-1,s|s+1|\dots|i+1,i+2 | \dots| d+1]$. For $j<s$ the
    inequality of the $j$th entries is clear. For $j\ge s$, and $j\neq
    i$ the $a_j$ are again ordered, because we assume that $a_j \le
    a_{j+1}$ whenever~$j\ge s$. Their $i$th entries correspond to
    $\{i,i+1\}$ and $\{i+1,i+2\}$, thus as also $a_i\le a_{i+2}$ we
    conclude.
  \item If $i=d+1$, then $[1|2|\dots| d-1|\widehat{d,d+1}] \le_{\aaa}
    [1|2|\dots|d-1,d|\widehat{d+1}]$ for any $\aaa$ and we conclude by
    the previous case. \qedhere
 \end{itemize}
\end{proof}
Example~\ref{ex:alex-method} reproduced the above construction in the
case~$d=s=2$.  We are now ready to prove the third and most
complicated case.

\begin{proposition}\label{prop:d+2partite}
  If $G_\D$ is $q$-partite with $q\ge3$ and has at least two parallel
  classes of cardinality $\ge 2$, then the $h$-vector $h^\D$ is a pure
  $O$-sequence.
\end{proposition}
 
 \begin{proof}
   The proof is a repeated application of
   Lemma~\ref{lem:compatibility} with the tripartite graph of
   Lemma~\ref{lem:induction basis} as the base case.  This is possible
   because of the two parallel classes of cardinality~$\ge 2$.  Order
   the vertices of $G_\D$ such that each parallel class contains
   consecutive vertices.  With this convention, there are only two
   cases to consider:
\begin{enumerate}[leftmargin=9ex]
\item[Case 1:] $d+2$ is parallel to $d+1$ in $G_\D$.
\item[Case 2:] $d+2$ is not parallel to any vertex is $G_\D$.
\end{enumerate}
We use the notation of Lemma~\ref{lem:compatibility} for $\G'$ and
$\G''$.
\paragraph{\textit{Case 1}} Let $\{r, \ldots, d+1, d+2\}$ be the
parallel class of $d+1$ in $G_\D$.  By Remark~\ref{rem:G<->Delta},
$\D\setminus A_{d+2} = \D_{d+2-r}(d,d+1,(a_1,\ldots,a_{d+1}))$, we can choose 
$\G' = \G_{d+2-r}(d,d+1,(a_1,\ldots,a_{d+1}))$.  The matroid $\lk_\D
A_{d+2}$ corresponds to $G_\D\setminus \{d+2\}$, thus by the inductive
hypothesis there exists an order ideal $\G''$ such that $h^{\lk_\D
  A_{d+2}}=f(\G'')$. We may also assume that $\G''$ was obtained by a
repeated application of Lemma~\ref{lem:compatibility}, and thus among
its generators has:
\[ 
[1|2|\dots|r-2,r-1|r|\dots|d,d+1],\dots, [1,2|3|\dots|
r-1|r|\dots|d,d+1].
\] 
These generators appear from generators of the $\G'$ at the previous
step because $\lk_{G_\D}(d+1)$ is isomorphic to $\lk_{G_\D}(d+2)$. Compatibility is
easy to confirm.

\paragraph{\textit{Case 2}} 
 Let $ \{r, \ldots, d+1\}$ be the parallel class of $d+1$
in~$G_\D$.  Define a permutation $\sigma$ of the vertices
of~$G_\D\setminus \{d+2\}$.  In order to not complicate notation more
than necessary, do this inductively on the parallel classes.  The
first two parallel classes remain unchanged. For every other parallel,
reverse the order of its vertices. More precisely, assume for every
$i<r$ that $\sigma$ is already defined.  For every  $ j \in \{r, \ldots, d+1\}$, set
$\sigma(j)=r+d+1-j$.  As $d+2$ is not parallel to any vertex in
$G_\D$, Remark~\ref{rem:G<->Delta} implies that the deletion
$\D\setminus A_{d+2}$ is $ \D_0(d,d+1,(a_1,\ldots,a_{d+1}))$.  Now use
\cite[Theorem 3.5]{CV3} with the vertices permuted by~$\sigma$. That
is we have $h^{\D\setminus A_{d+2} }=f(\G')$, with $\G'$ generated by:
\[\begin{array}{rcccccccccccccccl}
\lbrack&1&|&2&|&\dots&|&m&|&d+1&|&d&|&\dots&|&r+1,r&\rbrack\\
&&&&&&&&\vdots&&&&&&&&\\
\lbrack&1&|&2&|&\dots&|&m&|&d+1,d&|&d-1&|&\dots&|&r&\rbrack\\
\lbrack&1&|&2&|&\dots&|&m, d+1&|&d&|&d-1&|&\dots&|&r&\rbrack\\
&&&&&&&&\vdots&&&&&&&&\\
\lbrack&1,2&|&3&|&\dots&|& d+1&|&d&|&d-1&|&\dots&|&r&\rbrack,\\
\end{array}\]
for some $m$ which plays no role in the proof. 
Inductively construct $\G''$ such that $h^{\lk_\D A_{d+2}}=f(\G'')$.
Assume that $\G''$ was constructed using the same strategy of
permuting and applying Lemma~\ref{lem:compatibility} just with
$(r-1)$-compatibility.  For each $j = r+1,\dots,d+1$, there are $r-1$
generators of $\G''$ which have been added at the $j$th step. This is
due to the fact that the simplification of $\D|_{A_1,\dots,A_{j-1}}$
is dual to the discrete matroid on $j-1$ vertices with $j-r$ loops,
thus its $h$-vector is obtained from
$\G_{j-r}(j-1,j,(a_{\sigma(1)},\dots,a_{\sigma(j)}))$. After applying
the gluing from Lemma~\ref{lem:compatibility}, the generators are:
\[\begin{array}{rcccccccccccccccccl}
\lbrack&1&|&2&|&\dots&|&m',m&|&d+1&|&\dots&|&j,j-1&|&\dots&|&r&\rbrack\\
&&&&&&&&&\vdots&&&&&&&&&\\
\lbrack&1,2&|&3&|&\dots&|&m&|&d+1&|&\dots&|&j,j-1&|&\dots&|&r&\rbrack,\\
\end{array}\]
where $m$ and $m'$ depend on the cardinality of the parallel class of
$r-1$ in~$G_\D$. Their precise description is not needed, as they take
the same values for both $\G''$ and $\G'$.

To check $(r-1)$-compatibility, let $P =
[1|2|\dots|\sigma(i),\sigma(i+1)|\dots|r]$ be a generator of~$\G'$.
If $i<r-1$, then choose $Q$ among the generators added at the
$(d+1)$th step, namely
\[
Q= [1|2|\dots|\sigma(i),\sigma(i+1)|\dots|m|d+1,d| \dots|r].
\]
If $i>r-1$, then choose $Q$ among the generators added at the
$\sigma(i)$th step, namely
\[Q = [1|2|\dots|m',m|d+1|\dots|\sigma(i),\sigma(i+1)|\dots|r].
\]
It is easy to see that in both cases $P\le_{\aaa}^{r-1} Q$ for any
vector~$\aaa$.  Finally, the proof of Case 1 works identically also if
$\sigma$ is applied to the inductive hypothesis.  
\end{proof}
Propositions \ref{prop:case1}, \ref{prop:bipartite}, and
\ref{prop:d+2partite}, together with the $(d+1)$-partite case
\cite[Corollary 3.9]{CV3} imply the main theorem of this section.
\begin{theorem}\label{thm:d2}
  If $\D$ is a rank $d$ matroid with at most $d+2$ parallel classes,
  then the $h$-vector of the quotient by its cover ideal is a pure
  $O$-sequence.
\end{theorem}

\section{Small type}
\label{sec:type}
If $h^{\Delta} = h_{\Dc}$ is the $h$-vector of the cover ideal
of a matroid $\Delta$, then its last entry is the Cohen-Macaulay type
of~$\kk[\Dc]$.  If it is small, then the parallel classes of the
matroid must be few thanks to \cite[Remark 4.4]{CV3}: Precisely, if a
matroid is of rank $d$ and has $p$ parallel classes, then its type is
at least $p-d+1$.  Theorem~\ref{thm:type5} exploits this fact to prove
that $h^{\Delta}$ is a pure $O$-sequence whenever the type is at most
five.  We start with a proposition that shows that among the simple
matroids there is only one of rank $d$ with $p$ parallel classes and
whose type is $p-d+1$.

\begin{proposition}\label{prop:typebound}
  Let $\D$ be a $p$-partite matroid of rank $d$. Then
  $\type(S/J(\D))=p-d+1$ if and only if $\Ds=\D_{d-2}(d,p,{\bf 1})$.
\end{proposition}
\begin{proof}
  By \cite[Proposition 2.8]{CV3} we can assume that $\D$ is simple.
  \cite[Remark~4.4]{CV3} shows that $\type(S/J(\D))\geq p-d+1$, and
  equality holds if $\D=\D_{d-2}(d,p,{\bf 1})$.  Assume that $\Delta$
  satisfies $\type(S/J(\D)) = p-d+1$.  The proof is by induction on
  $p-d$.  The base case is when $d=p$ in which case $\Ds$ is a
  simplex.  Now assume that $p-d$ is positive.  Without loss of
  generality assume that the vertex $p$ is not a cone point (otherwise
  relabel the vertices).  By \cite[Remark~1.7]{CV3} we have
  \[
  h_k^\D=h_{k-1}^{\D\setminus p}+h^{\lk_{\D}p}_{k}\ \ \ \forall \
  k\in \ZZ.
  \]
  Again by \cite[Remark~4.4]{CV3} and since $\type(S/J(\D))=p-d+1$, we
  get $\type(S/J(\D\setminus p))=p-d$ and $\type(S/J(\lk_{\D}p))=1$.
  The matroid $\lk_{\D}p$ is $(d-1)$-partite and, by the induction
  hypothesis, $\D \setminus p=\D_{d-2}(d,p-1,{\bf 1})$.  After
  potentially relabeling the vertices, $\{1,2,\ldots ,d-2,i,j\}$ is a
  face of $\D$ for all $i,j\in\{d-1,\ldots ,p-1\}$.  If $\{1,2,\ldots
  ,d-2,p\}$ was not a face of $\D$, then 
  there is some $k\in\{1,\ldots ,d-2\}$ such that $\{1,\ldots
  ,\hat{k},\ldots , d-2,i,j,p\}$ is a face of $\D$ for all $i$ and $j$
  in $\{d-1,\ldots ,p-1\}$.  This would imply that $\{i,j\}\in
  \lk_{\D}p$ for all $i,j\in\{1,\ldots , p-1\}\setminus\{k\}$ and
  $\lk_{\D}p$ would be $(p-2)$-partite---a contradiction.  Therefore
  $\{1,2,\ldots ,d-2,p\}$ is a face of $\D$.  We now show that, for
  fixed $i\in\{d-1,\ldots ,p-1\}$, the set $\{i,k\}$ is a face of
  $\lk_{\D}p$ for all $k\in\{1,\ldots ,d-2\}$.  If not, then
  $\{1,\ldots ,d-2,j,p\}$ is a facet of $\D$ for all $j\in\{d-1,\ldots
  ,p-1\}\setminus\{ i\}$.  Pick $r,s\in\{d-1,\ldots ,p-1\}\setminus
  \{i\}$.  Certainly $B=\{1,\ldots ,d-2,r,s\}$ is a facet of $\D$.
  Since $i$ is parallel to some $k\in\{1,\ldots ,d-2\}$ (in
  $\lk_{\D}p$), also $B'=(\{1,\ldots ,d-2,r,p\}\setminus\{k\})\cup
  \{i\}$ is a facet of~$\D$.  Removing $k$ from $B$, the only way to
  satisfy basis exchange among $B$ and $B'$ is that $\{r,s,p\}$ is a
  face of~$\D$.  In this case, however, $\lk_{\D}p$ would be
  $d$-partite, since the restriction of its 1-skeleton to the vertices
  $\{1,\ldots ,d-2,r,s\}$ would be a complete graph.
\end{proof}

\begin{remark}
  Theorem~4.3 in \cite{CV3} says that $h^{\Delta_{d-2}(d,p,{\bf 1})}$
  is a componentwise lower bound for all simple matroids of rank $d$
  on $p$ vertices.
\end{remark}
 
\begin{theorem}\label{thm:type5}
  Let $\Delta$ be a matroid and $h^{\Delta} = (1,h_1,\ldots, h_s)$ its
  $h$-vector.  If $h_s\leq 5$, then $h^{\Delta}$ is a pure
  $O$-sequence.
\end{theorem}
\begin{remark}
  By duality, Theorem~\ref{thm:type5} also holds for Stanley-Reisner
  ideals.
\end{remark}
\begin{proof}[Proof of Theorem~\ref{thm:type5}]
  By \cite[Remark~4.4]{CV3} $\type(S/J(\D))\geq p-d+1$ which in our
  case implies $p \leq d+4$.  The cases $p = d$ and $p=d+1$ are
  trivial, and $p=d+2$ is the content of Theorem~\ref{thm:d2}.
  By Proposition~\ref{prop:typebound}, if $p=d+4$, then
  $\Ds=\D_{d-2}(d,p,{\bf 1})$ and the result follows from
  \cite[Theorem~3.7]{CV3}.  It remains to check the case $p=d+3$,
  however, there are no simple matroids with cover ideal of type five
  such that $p=d+3$.  To see this, assume $\Delta$ is such a matroid
  and consider its dual~$\Dc$.  The simplification $\Ds$ has
  the same type, so we can assume that $\Delta$ is simple and
  consequently $\Dc$ is of rank three.  Let $G$ be the complete
  $q$-partite graph which is the 1-skeleton of~$\Dc$.  Since $\Dc$ is
  of rank three, $q\geq 3$. Let $b_{1}\geq
  \dots \geq b_{q}$ be the sizes of the parallel classes in~$G$ which
  we can assume ordered nonincreasingly.  
  Let $h^{\Dc}=(1,h_{1},h_{2},5)$ be the
  $h$-vector.  By the Brown-Colbourn
  inequalities~\cite[Theorem~3.1]{brown1992roots}, $1-h_{1}+h_{2} \leq
  5$.  If $n \leq d+3$ is the number of vertices of $G$ and $e$ the
  number of edges, then $h_{1} = n-3$ and $h_{2}=3-2n+e$.  It follows
  that $e\leq 3n-2$.  Now, if $q=3$, then $b_{i} \geq 3$ for
  $i=1,\dots,q$ and $e>3n-2$.  If $q=4$, then $b_{i}\geq 2$ for
  $i=1,\dots,q$, except for one graph in which $b_{4} = 1$ and
  $b_{2}=b_{3}=b_{4}=2$.  If $q=5$, there are five possible graphs.
  If $q=6$, then $K_{6}$ the complete graph is the only possible
  graph.  When the graph is fixed, the $h$-vector of $\Dc$ is
  fixed.  Table~\ref{tab:graphs} summarizes the possible graphs and
  their $h$-vectors.
  \begin{table}[htpb]
    \centering
    \begin{tabular}{|c|c|c|}
      \hline
      \raisebox{.5ex}{$q$} & \raisebox{.3ex}{$(b_{1},\dots,b_{q})$} &\rule{0pt}{2.7ex} $h^{\Delta}$ \\
      \hline\hline
      4 & $(2,2,2,1)$      & $(1,4,7,5)$ \\
      5 & $(1,1,1,1,1)$    & $(1,2,3,5)$ \\
      5 & $(2,1,1,1,1)$    & $(1,3,5,5)$ \\
      5 & $(2,2,1,1,1)$    & $(1,4,8,5)$ \\
      5 & $(3,1,1,1,1)$    & $(1,4,7,5)$ \\
      5 & $(4,1,1,1,1)$    & $(1,5,9,5)$ \\
      6 & $(1,1,1,1,1,1)$  & $(1,3,6,5)$ \\
      \hline
    \end{tabular}\vspace*{1ex}
    \caption{\label{tab:graphs} Possible $q$-partite graphs in the proof of Theorem~\ref{thm:type5}}
  \end{table}
  Using the database of Mayhew and Royle~\cite{MR}, a simple for-loop
  in Sage enumerates all matroids of rank three, filters those with
  the given $h$-vectors, computes their duals, and confirms that none
  is simple.
\end{proof}

\begin{remark}
  The matroid $\D_{d-2}(d,p,{\bf 1})$ is the only matroid of type $t$
  which satisfies $p = d+t-1$ and in the proof of
  Theorem~\ref{thm:type5} we showed that, if $t=5$, then there is no
  matroid of type five such that $p=d+3$.  It would be interesting to
  understand for which $t$ there is a such a gap in the allowable
  number of parallelism classes.
\end{remark}

\section{The search for counterexamples}
\label{sec:computer}
As soon as the number of variables $d$, the socle degree $s$, and the
type $t$ are fixed, one can enumerate all pure $O$-sequences with
these characteristics.  A pure order ideal with these data is
generated by $t$ monomials of degree~$s$.  Let $N_{d,s} =
\binom{s+(d-1)}{d-1}$ be the number of monomials of degree $s$ in $d$
variables.  A priori, there are $\binom{N_{d,s}}{t}$ generating sets
of order ideals to consider and our program loops over these,
computing their $f$-vectors.  Naturally, many of those socles will be
equivalent after relabeling the variables, or have the same $f$-vector
even if they are not equivalent.  One may hope to reduce the number of
combinations by exploiting this symmetry.  However, it is not clear
how to do so.  Checking if two socles are equivalent after permuting
the variables is computationally more expensive than just computing
the $f$-vectors of the order ideals they generate.  One shortcut that
is easy to implement is to require the lexicographically first
monomial in each socle to have weakly increasing exponent vector.
This can be achieved by a permutation of the variables and is quick to
check.  Further improvements are possible if one is not interested in
all pure $O$-sequences, but just wants to check a particular example.
The computation of the face numbers of an order ideal descends degree
by degree.  In each step, the program searches for monomials that
divide the given monomials in the previous degree.  If a candidate
$h$-vector is given, then one can stop the degree descent as soon as
there is disagreement between the candidate vector and the number of
monomials in the current degree.  Our software implements all of these
shortcuts.

\begin{example}\label{ex:doubleBinomial}  By Theorem~\ref{thm:type5}
  and~\cite{DKK} any candidate counterexample for Stanley's conjecture
  must be on at least ten vertices and of Cohen-Macaulay type six.
  Assume that $\Delta$ is of rank four.  For $h$-vectors of cover
  ideals, checking an example with this data amounts to enumerating
  order ideals generated by six monomials of degree six in four
  variables.  Our implementation handles approximately 30000 order
  ideals per second on a standard laptop.  Checking all $\binom{84}{6}
  = 406481544$ potential socles would take approximately four hours.
  However, this number grows quickly.  If a counterexample exists and
  was of rank five on twelve vertices and type seven, then a
  back-of-the-envelope calculation estimates the computational time as
  around 173 CPU years.
\end{example}

Lemma~\ref{lem:compatibility} inspires a method to search for pure
order ideals.
\begin{method} \label{meth:alex} Let $\Delta$ be a $p$-partite matroid
  of rank $d$ with parallel classes $A_{1},\dots,A_{p}$ which we may
  choose ordered such that $A_1\ldots A_d \in \D$, that is
  $\{v_{1},\dots,v_{d}\}$ is a facet whenever $v_{i}\in A_{i}$ for all
  $i=1,\dots,d$.  To find a pure order ideal whose $f$-vector
  equals~$h^{\D}$, instead of enumeration, one may proceed as follows.
  \begin{enumerate}[leftmargin=4ex]
  \item\label{one} For each $i\in\{d,\ldots,p\}$ let $G_i$ be the set
    of generators of $\G_0(d-1,i-1,(a_1,\ldots,a_{i-1}))$.
  \item\label{two} Compute $c_i$, the last entry of the $h$-vector of
    $\lk_{\D|_{A_1\cup\dots\cup A_{i-1}}} A_i$.
  \item\label{three} For every $i\in\{d,\ldots,p\}$ choose a
    $c_i$-subset $H_i$ of $G_i$.
  \item\label{four} Define $\G =\< \overline{H}_d\cup\ldots\cup
    \overline{H}_p \>$, where the collection of partitions
    $\overline{H}_j$ is obtained by adding the set $\{j,\ldots,p\}$ to
    every $(d-1)$-partition of $[j-1]$ contained in $H_j$.
  \item\label{five} Check if $h^\D = f(\G)$.
\end{enumerate}
\end{method}

The gist of this method is, instead of searching all socles, to only
search order ideal generators among the monomials that could
potentially arise from a repeated application of
Lemma~\ref{lem:compatibility}.  The method starts at the complete
matroid $\D|_{A_1\cup\dots\cup A_d}$ and reconstructs $\D$ by gluing
the remaining parallel classes.  In this process it mimics the
construction of Lemma~\ref{lem:compatibility} in many different ways.
The compatibility condition is never checked.  It is faster to just
confront the $f$-vector of the final result with~$h^\D$.

The choice of ordering of the $A_{i}$ fixes the order in which
Lemma~\ref{lem:compatibility} would be applied (and one may try
different orderings).  Step~\eqref{one} creates lists of candidates
for the generators of $\G''$ (in the notation of the lemma).
Steps~\eqref{two} and~\eqref{three} enumerate the sets of order ideal
generators that may result from the choices.  Finally,
Step~\eqref{four} implements the gluing in
Lemma~\ref{lem:compatibility}.  Evidently, if the procedure does not
find an order ideal whose $f$-vector is $h^{\Delta}$ we have not found
a counterexample.
 
\begin{example}
  In specific examples, the number of orderings of the parallel
  classes can be reduced using symmetries of the matroid.  For
  instance in Example~\ref{ex:alex-method} the pairs $(A_1,A_2)$ and
  $(A_3,A_4)$, and also the classes in each pair, could be exchanged.
  Given that $A_1A_2A_5$ and $A_3A_4A_5$ are not in $\D$, the only
  orderings to check in this case are $A_1,A_2,A_3,A_4,A_5$ and
  $A_1,A_3,A_5,A_2,A_4$.
\end{example}
\begin{example}\label{ex:smartsearch}
  Let $\Delta$ be the simple rank four matroid on eight vertices with
  the following facets:
  \begin{gather*}
    1235, 1236, 1237, 1238, 1245, 1246, 1247, 1248, 1256, 1257, 1268,
    1278, 1345, 1346, 1347, \\ 1348, 1357, 1358, 1367, 1368, 1456,
    1458, 1467, 1478, 1567, 1568, 1578, 1678, 2356, 2357, \\ 2358,
    2456, 2457, 2458, 2568, 2578, 3456, 3457, 3458, 3567, 3568, 4567,
    4578,5678.
  \end{gather*}
  Precisely, $\Delta$ is a series-extension ($15$ is a cocircuit) of
  the Fano matroid.  The largest example that we tried our method on
  is the rank four matroid $\Delta_{\aaa}$ on 20 vertices whose
  simplification is $\Delta$ and whose parallel classes have sizes
  $(1,2,3,4,1,3,4,2)$.  We have
  \[ 
  h^{\Delta_{\aaa}} = (1, 4, 9, 16, 25, 36, 49, 64, 81, 100, 112, 116,
  111, 96, 70, 40, 14)
  \] 
  which means that enumeration of order ideals is entirely pointless.
  However, using Method~\ref{meth:alex} we found that this vector is a
  pure $O$-sequence.  It equals the $f$-vector of the order ideal
  \begin{gather*}
    \Gamma = \< bc^{2}d^{13}, bc^6d^9, b^4c^3d^9, bc^{10}d^5, b^8c^3d^5,
    bc^{12}d^3, b^4c^9d^3, \\ \qquad \quad a^9b^3c^4, a^5b^9c^2, bc^{15}, a^5b^3c^8,
    b^{14}c^2, a^2b^{12}c^2, a^2b^{10}c^4 \>.
  \end{gather*}
  The Artinian monomial level algebra with $\kk$-basis $\Gamma$ is
  $\kk[a,b,c,d]/I$ where
  \begin{gather*}
    I = \big(
    a^{10},a^6b^4,a^3b^{10},ab^{13},b^{15},a^3b^4c^3,b^{11}c^3,a^6c^5,ab^4c^5,b^5c^5,ac^9,b^2c^{10},\\
    \qquad \quad
    c^{16},ad,b^9d,b^5c^4d,c^{13}d,b^2c^4d^4,c^{11}d^4,b^5d^6,c^7d^6,b^2d^{10},c^3d^{10},d^{14}
    \big).
  \end{gather*}
\end{example}

\begin{remark}
  The number of different $h$-vectors of coloop free matroids is equal
  to the number of different $f$-vectors of coloop free matroids.
  Since matroids are very particular pure multicomplexes, the number
  of their $f$-vectors is smaller than the number of pure
  $O$-sequences (which are $f$-vectors of pure multi-complexes).
  Therefore, it seems plausible that the probability of finding a pure
  $O$-sequence equal to the $h$-vector of a matroid tends to zero as
  the parameters grow.  This limits the usefulness of random search
  for order ideals in larger examples.
\end{remark}

\bibliographystyle{amsplain}
\bibliography{stanley}
\enlargethispage{10ex}
 \end{document}